\documentclass{amsart}

\usepackage{graphics}
\usepackage{amsmath,amsthm}
\usepackage{amssymb}
\usepackage{amsfonts}
\usepackage{tikz}
\usetikzlibrary{tqft}
\usetikzlibrary{shapes}
\usepackage{cite}
\usepackage{graphics}
\usepackage{float}

\newtheorem{theorem}{Theorem}[section]
\newtheorem{proposition}[theorem]{Proposition}%

\newtheorem{corollary}[theorem]{Corollary}
\newtheorem{lemma}[theorem]{Lemma}
\theoremstyle{definition}
\newtheorem{remark}[theorem]{Remark}
\newtheorem{example}[theorem]{Example}
\newtheorem{definition}[theorem]{Definition}

\newcommand{\K}{\Bbbk}

\newcommand{\im}{\operatorname{im}}
\newcommand{\tr}{\operatorname{tr}}
\newcommand{\str}{\operatorname{str}}
\newcommand{\id}{\operatorname{id}}
\newcommand{\gr}{\operatorname{gr}}
\newcommand{\End}{\operatorname{End}}
\newcommand{\Pair}{\mathbf{Pair}_\K}
\newcommand{\Unc}{\mathbf{Uncurved}_\K}

\renewcommand{\geq}{\geqslant}

\begin{document}

\title{The Supertrace Defect of a Graded Endomorphism}

\author{Esma Dirican Erdal}
\address{Department of Mathematics, I\c{s}{\i}k University, \.{I}stanbul, 34980, T\"{u}rkiye}\thanks{This work was partially supported by
ITU-DOSAP (Proj.ID:43959)}
\email{esma.diricanerdal@isikun.edu.tr}

\author{Atabey Kaygun}
\address{Department of Mathematics Engineering, \.{I}stanbul Technical University, \.{I}stanbul, 34467, T\"{u}rkiye}
\email{kaygun@itu.edu.tr}

\date{\today}

\dedicatory{}

\commby{}


\begin{abstract}
Motivated by Igusa's relation between relative cyclic homology and the logarithm of the
graded Cartan determinant, we isolate the homological step behind the passage from
logarithmic determinants to traces of powers.  For a bounded chain complex $(C_*,d)$ and
a degree-zero graded endomorphism $f$, not assumed to commute with $d$, we construct the
universal quotient on which $f$ becomes a chain map.  The kernel of this reflection
iterates to a canonical functorial minimal filtration, whose Rees module gives a flat
degeneration to an uncurved associated graded.  For finite-dimensional chain groups all
power supertraces of $f$ are represented by Lefschetz traces on this associated graded,
so the resulting logarithmic-determinant and zeta series admit a canonical homological
representation.
\end{abstract}


\maketitle

\section{Introduction}

This paper grew out of an attempt to understand the mechanism behind Igusa's relation
between cyclic homology and the Cartan determinant.  For a graded finite-dimensional
algebra, Igusa relates the Euler characteristic of relative cyclic homology of the
radical to the logarithm of the determinant of the graded Cartan matrix
\cite{Igusa1992}.  At the formal level, logarithmic determinants are governed by traces
of powers through the identity
\[
  -\log\det(1-tT)=\sum_{r\geq1}\frac{\tr(T^r)}{r}t^r.
\]
Thus one of the basic transitions in such an argument is from determinant data to an
additive sequence of power traces, followed by a homological interpretation of those
traces.  The present paper isolates this intermediate step.

Our concern is deliberately narrower than Igusa's theorem.  We do not identify the
specific relative cyclic complex occurring in the Cartan-determinant problem, nor do we
address the divisor and M\"obius transforms which enter the cyclic formula.  Instead we
ask a prior question.  Suppose $(C_*,d)$ is a bounded chain complex and $f$ is a
degree-zero graded endomorphism which need not commute with $d$.  The power traces
$\str(f^r)$ are still defined when the chain groups are finite-dimensional, but the Hopf
trace formula cannot be applied because $f$ need not induce an endomorphism of homology.
Can one replace $(C_*,f)$ canonically by homological data which retain the same power
traces?

The answer begins with a universal one-step construction.  Let $\Pair$ be the category
of pairs $(C_*,f)$ and equivariant chain maps, and let $\Unc\subset\Pair$ be the full
subcategory of \emph{uncurved} pairs, those for which $f$ is a chain map.  We associate
to $(C_*,f)$ the smallest $f$-stable subcomplex containing the commutator image,
\[
  \Delta(C,f)=\sum_{j\geq0}f^j(\im[d,f]),
\]
and show that
\[
  Q(C,f):=C_*/\Delta(C,f)
\]
is the universal quotient on which $f$ becomes a chain map.  Consequently the
assignment $F(C,f)=(Q(C,f),\widetilde f)$ is left adjoint to the inclusion
$I\colon\Unc\hookrightarrow\Pair$.  In this sense the failure of $f$ to commute with the
differential has a canonical maximal quotient on which it disappears.

The adjunction itself produces the higher uncurving structure.  If
$\eta_X\colon X\to IFX$ is the unit and $KX:=\ker\eta_X$, then
$KX=\Delta(C,f)$ for $X=(C,f)$.  Iterating this kernel gives a canonical tower
\[
  X=G^0X\supseteq G^1X\supseteq G^2X\supseteq\cdots,
  \qquad G^kX=K^kX,
\]
with
\[
  \gr_G^kX\cong IF(K^kX).
\]
Thus every layer is uncurved by construction.  For bounded complexes the tower is
finite, functorial, and minimal among all filtrations whose successive quotients are
uncurved.  The filtration is therefore not an auxiliary choice: it is forced by the
unit of the reflection.

The same tower has a deformation-theoretic form.  Its Rees module
\[
  R_G(C):=\bigoplus_{k\in\mathbb Z}G^kC\,u^{-k}
          \subseteq C_*[u,u^{-1}]
\]
has fiber $(C,f)$ at $u=1$ and the uncurved associated graded at $u=0$, while
$[d,f]R_G(C)\subseteq uR_G(C)$.  For the canonical kernel tower this interpolation is
functorial.  Thus the reflection, its kernel tower, and the Rees degeneration are three
forms of the same canonical procedure: remove the obstruction once, resolve what
remains layer by layer, and package the resolution into a flat family.

For finite-dimensional chain groups, this construction gives the homological trace
statement needed at the logarithmic-determinant interface.  Trace additivity along any
uncurving filtration yields
\[
  \str_{C_*}(f)
  =\sum_k L\bigl(H_*(\gr_G^k f),H_*(\gr_G^k C_*)\bigr),
\]
and for the canonical tower the right-hand side is the Lefschetz number of a
canonically associated chain endomorphism.  More importantly for the motivating
problem, the same statement applies to every power $f^r$.  Hence
\[
  \log\!\left(\prod_n\det(1-tf_n)^{(-1)^{n+1}}\right)
  =\sum_{r\geq1}\frac{\str(f^r)}{r}t^r
\]
can be rewritten entirely in terms of Lefschetz traces of the uncurved associated
graded.  Equivalently, the power-supertrace series of an arbitrary graded endomorphism
is represented by an ordinary Lefschetz zeta function after canonical uncurving.

This is the precise point at which the present paper interfaces with the motivation from
Igusa.  Igusa's theorem continues beyond this point by identifying the relevant trace
data with the Euler characteristic of a relative cyclic complex and by applying the
arithmetic transform appropriate to the Cartan determinant.  Those cyclic and
arithmetic steps are not part of the present paper.  Our purpose is to isolate a reusable
homological mechanism which explains how trace-of-power data can first be moved into a
setting where homology, Lefschetz theory, and logarithmic determinants can interact.

The classical Hopf trace formula and Euler--Poincar\'e identity appear as limiting cases
of this mechanism.  The examples also show that the construction is not confined to the
motivating Cartan-determinant problem: it produces Lefschetz zeta functions, recovers
Banchoff's vertex index and its Reidemeister-torsion refinement, and realizes the
relative de Rham complex as a universal uncurved quotient.

Section~\ref{sec:reflection} constructs the universal one-step uncurving and proves the
reflection theorem.  Section~\ref{sec:tower} identifies the canonical uncurving
filtration with the kernel tower of the unit and proves finite termination and
minimality.  Section~\ref{sec:rees} develops the Rees deformation.  Section~\ref{sec:trace}
derives the trace consequences, with the Hopf trace formula and Euler--Poincar\'e as
corollaries.  Section~\ref{sec:examples} gives the zeta, torsion, and geometric examples.

\subsection*{Conventions}

Let $\K$ be a field of characteristic zero.  Chain complexes are homologically graded,
$d_n\colon C_n\to C_{n-1}$.  The structural constructions require boundedness but not
finite-dimensional chain groups; finite-dimensionality is imposed when traces or
determinants are used.  Graded maps obey the Koszul sign rule.  For a finite-dimensional
complex, if $X\subseteq V$ is invariant under an endomorphism $u$, we write
$\tr(u,X)$ for the trace of $u|_X$ and
\[
  L(H_*(g),H_*(C_*)):=\sum_n(-1)^n\tr(H_n(g))
\]
for the Lefschetz number of a chain endomorphism $g$.

\section{Universal One-Step Uncurving}\label{sec:reflection}

\begin{definition}\label{def:pair}
Let $\Pair$ be the category whose objects are pairs $(C_*,f)$ consisting of a bounded
chain complex and a degree-zero graded endomorphism $f$ of $C_*$.  A morphism
$\phi\colon(C_*,f)\to(D_*,g)$ is a chain map satisfying $\phi f=g\phi$.  A pair is
\emph{uncurved} if $f$ is a chain map; $\Unc\subset\Pair$ denotes the full subcategory
of uncurved pairs.
\end{definition}

For a pair $(C_*,f)$ write $[d,f]=df-fd$.

\begin{lemma}\label{lem:identities}
For every $\ell\geq1$,
\begin{align}
  d[d,f^\ell]&=-[d,f^\ell]d,\label{eq:d-commutator}\\
  [d,f^\ell]&=\sum_{j=0}^{\ell-1}f^j[d,f]f^{\ell-1-j},\label{eq:commutator-power}\\
  f[d,f^\ell]&=[d,f^{\ell+1}]-[d,f]f^\ell.\label{eq:f-stability}
\end{align}
In particular, $\im[d,f^\ell]$ is a subcomplex of $C_*$.
\end{lemma}

\begin{proof}
The first identity follows by expanding and using $d^2=0$.  The second is the usual
commutator expansion, and the third follows by expanding its two sides.
\end{proof}

\begin{definition}\label{def:defect}
For $X=(C_*,f)\in\Pair$ define
\begin{equation}\label{eq:defect-span}
  \Delta(X)=\Delta(C,f):=\sum_{j\geq0}f^j(\im[d,f]),
  \qquad
  Q(X)=Q(C,f):=C_*/\Delta(C,f).
\end{equation}
\end{definition}

\begin{proposition}[Universal defect subcomplex]\label{prop:defect}
For every pair $(C_*,f)$ one has
\begin{equation}\label{eq:defect-equivalent}
  \Delta(C,f)=\sum_{\ell\geq1}\im[d,f^\ell].
\end{equation}
The space $\Delta(C,f)$ is the smallest $f$-stable subcomplex of $C_*$ containing
$\im[d,f]$.  Consequently $f$ induces a chain endomorphism $\widetilde f$ of $Q(C,f)$.
Equivalently, $Q(C,f)$ is the largest quotient complex of $C_*$ to which $f$ descends
as a chain map.
\end{proposition}

\begin{proof}
Put $V=\im[d,f]$ and $B=\sum_{\ell\geq1}\im[d,f^\ell]$.  By
\eqref{eq:commutator-power}, $B\subseteq\sum_{j\geq0}f^jV=\Delta(C,f)$.  Conversely,
\eqref{eq:f-stability} shows that $fB\subseteq B$, and $V\subseteq B$; hence
$\Delta(C,f)\subseteq B$.  This proves \eqref{eq:defect-equivalent}.  Each summand in
$B$ is a subcomplex by \eqref{eq:d-commutator}, so $\Delta(C,f)$ is a subcomplex.  Its
definition makes it the smallest $f$-stable subspace containing $V$, hence the smallest
such subcomplex.  Finally $[d,f]C_*\subseteq\Delta(C,f)$, so the induced operators on
$Q(C,f)$ commute.
\end{proof}

\begin{theorem}[Universal uncurving and reflection]\label{thm:reflection}
For every $X=(C_*,f)\in\Pair$, the quotient map
\[
  \eta_X=q_X\colon X\longrightarrow (Q(C,f),\widetilde f)
\]
is universal among morphisms from $X$ to uncurved pairs.  Consequently
\[
  F\colon\Pair\longrightarrow\Unc,
  \qquad
  F(C_*,f)=(Q(C,f),\widetilde f),
\]
is left adjoint to the inclusion $I\colon\Unc\hookrightarrow\Pair$.  The unit is
$\eta=q$ and the counit is the identity.
\end{theorem}

\begin{proof}
Let $\psi\colon(C_*,f)\to(D_*,g)$ be a morphism with $(D_*,g)$ uncurved.  Since $\psi$
is a chain map and $\psi f=g\psi$,
\[
  \psi[d,f]=[d,g]\psi=0.
\]
The kernel of $\psi$ is an $f$-stable subcomplex, so Proposition~\ref{prop:defect}
implies $\Delta(C,f)\subseteq\ker\psi$.  Hence $\psi$ factors uniquely through $q_X$.
This is the universal property.  Applied to $q_Y\phi$ for a morphism $\phi\colon X\to Y$,
it defines a unique morphism $F(\phi)\colon F(X)\to F(Y)$; uniqueness gives functoriality.
If $X$ is already uncurved, then $\Delta(X)=0$ and $q_X=\id_X$, so $FI=\id_{\Unc}$ and
the counit is the identity.  Thus $F\dashv I$.
\end{proof}

\begin{remark}\label{rem:reflection-meaning}
The reflector is idempotent and $q_X$ is an isomorphism exactly when $X$ is uncurved.
The kernel of $q_X$ is the entire obstruction that must be removed in any quotient on
which the endomorphism becomes a chain map.  The next section iterates precisely this
kernel.
\end{remark}

\section{The Kernel Tower of the Reflector}\label{sec:tower}

Let $\eta\colon\id_{\Pair}\to IF$ be the unit of the reflection of
Theorem~\ref{thm:reflection}.

\begin{definition}[Reflector kernel and canonical tower]\label{def:kernel-tower}
Define an endofunctor $K\colon\Pair\to\Pair$ by
\[
  KX:=\ker(\eta_X).
\]
For $X=(C_*,f)$, Proposition~\ref{prop:defect} identifies
\[
  KX=(\Delta(C,f),f|_{\Delta(C,f)}).
\]
The \emph{canonical uncurving filtration} of $X$ is
\[
  G^kX:=K^kX,
  \qquad k\geq0.
\]
\end{definition}

Naturality of $\eta$ makes $K$ a functor: a morphism $\phi\colon X\to Y$ restricts to
$K\phi\colon KX\to KY$.

\begin{proposition}[Canonical uncurved layers]\label{prop:canonical-layers}
For every $k\geq0$ there is a natural short exact sequence of pairs
\begin{equation}\label{eq:kernel-layer}
  0\longrightarrow G^{k+1}X
   \longrightarrow G^kX
   \xrightarrow{\eta_{G^kX}} IF(G^kX)
   \longrightarrow0.
\end{equation}
Consequently
\begin{equation}\label{eq:layer-reflector}
  \gr_G^kX:=G^kX/G^{k+1}X\cong IF(G^kX),
\end{equation}
so every associated-graded layer is uncurved.
\end{proposition}

\begin{proof}
By definition $G^{k+1}X=K(G^kX)=\ker\eta_{G^kX}$, while the unit is the quotient map
onto $IF(G^kX)$.  This is exactly \eqref{eq:kernel-layer}, and
\eqref{eq:layer-reflector} follows.
\end{proof}

\begin{lemma}[Finite termination]\label{lem:termination}
If $C_*$ is bounded, the canonical tower terminates after at most the number of degrees
$n$ for which $C_n\neq0$.
\end{lemma}

\begin{proof}
If $n$ is the largest degree with $(G^kC)_n\neq0$, then
$[d,f]\colon(G^kC)_{n+1}\to(G^kC)_n$ has zero source.  Hence
$\Delta(G^kC,f|_{G^kC})_n=0$ by \eqref{eq:defect-span}, so $G^{k+1}C$ has strictly
smaller top nonzero degree whenever $G^kC\neq0$.  Boundedness therefore forces
termination within the stated number of steps.
\end{proof}

\begin{definition}\label{def:uncurving}
An \emph{uncurving filtration} of $X=(C_*,f)$ is a finite decreasing filtration
\[
  C_*=H^0\supseteq H^1\supseteq\cdots\supseteq H^m=0
\]
by $f$-stable subcomplexes such that each quotient $H^k/H^{k+1}$ is uncurved.
Equivalently,
\[
  [d,f]H^k\subseteq H^{k+1}
  \qquad\text{for every }k.
\]
\end{definition}

\begin{proposition}[Minimality]\label{prop:minimal}
Let $G^\bullet$ be the canonical kernel tower and let $H^\bullet$ be any uncurving
filtration of $X$.  Then
\[
  G^kX\subseteq H^k
  \qquad\text{for every }k.
\]
Thus the canonical tower is the smallest uncurving filtration term by term.
\end{proposition}

\begin{proof}
The assertion is clear for $k=0$.  Suppose $G^kX\subseteq H^k$.  The composite
\[
  G^kX\longrightarrow H^k\longrightarrow H^k/H^{k+1}
\]
is a morphism to an uncurved pair.  By the universal property of
$\eta_{G^kX}\colon G^kX\to IF(G^kX)$, its kernel contains
$K(G^kX)=G^{k+1}X$.  Since the kernel of the composite is
$G^kX\cap H^{k+1}$, we obtain $G^{k+1}X\subseteq H^{k+1}$.
\end{proof}

\begin{corollary}[Functoriality]\label{cor:tower-functorial}
A morphism $\phi\colon X\to Y$ in $\Pair$ induces natural morphisms
\[
  G^k(\phi)=K^k(\phi)\colon G^kX\longrightarrow G^kY
\]
and hence morphisms of uncurved pairs
\[
  \gr_G^k(\phi)\colon\gr_G^kX\longrightarrow\gr_G^kY.
\]
\end{corollary}

\begin{proof}
This is immediate from functoriality of $K$ and the natural exact sequences
\eqref{eq:kernel-layer}.
\end{proof}

\begin{remark}[Uncurving depth]\label{rem:depth}
The least $m$ with $G^mX=0$ is the \emph{uncurving depth} of $X$.  It is at most one
exactly when $X$ is uncurved.  It is at most two exactly when the restriction of $f$ to
$\Delta(C,f)$ is a chain map.
\end{remark}

\section{The Rees Deformation}\label{sec:rees}

Let $H^\bullet$ be an uncurving filtration of $(C_*,f)$, extended by $H^k=C_*$ for
$k<0$ and by $H^k=0$ beyond its final term, and let $u$ be an indeterminate of degree
zero.  Define
\[
  R_H(C):=\bigoplus_{k\in\mathbb Z}H^k u^{-k}
          \subseteq C_*[u,u^{-1}].
\]

\begin{proposition}[Rees uncurving]\label{prop:rees}
For every bounded pair $(C_*,f)$ equipped with an uncurving filtration $H^\bullet$,
the Rees module $R_H(C)$ is a bounded complex of flat $\K[u]$-modules, and $f$ acts on
it $\K[u]$-linearly with
\[
  [d,f]R_H(C)\subseteq uR_H(C).
\]
There are natural isomorphisms of pairs
\begin{equation}\label{eq:rees-fibers}
  R_H(C)/(u-1)R_H(C)\cong(C_*,f),
  \qquad
  R_H(C)/uR_H(C)\cong(\gr_H C_*,\gr_H f).
\end{equation}
In particular the special fiber at $u=0$ is uncurved.  If, moreover, every $C_n$ is
finite-dimensional, then each $R_H(C)_n$ is a finitely generated free $\K[u]$-module
of rank $\dim_\K C_n$.
\end{proposition}

\begin{proof}
Since $d$ and $f$ preserve every $H^k$, they act $\K[u]$-linearly on $R_H(C)$.  The
uncurving condition gives
\[
  [d,f](H^ku^{-k})\subseteq H^{k+1}u^{-k}
   =u\,H^{k+1}u^{-(k+1)}\subseteq uR_H(C).
\]
For every $n$, the module $R_H(C)_n$ is a $\K[u]$-submodule of
$C_n[u,u^{-1}]$ and hence is torsion-free.  Since $\K[u]$ is a PID, every torsion-free
$\K[u]$-module is flat.  The quotient by $u$ is naturally
$\bigoplus_kH^k/H^{k+1}$, while $xu^{-k}\mapsto x$ gives the specialization at $u=1$
with kernel $(u-1)R_H(C)$.  If $C_n$ is finite-dimensional, the finite filtration
$H^\bullet C_n$ admits an adapted basis; the corresponding homogeneous generators show
that $R_H(C)_n$ is finitely generated.  Being finitely generated and torsion-free over
the PID $\K[u]$, it is free, and localization at $u$ identifies it with
$C_n[u,u^{-1}]$, so its rank is $\dim_\K C_n$.
\end{proof}

\begin{corollary}[Canonical Rees interpolation]\label{cor:rees-functorial}
For the canonical kernel tower $G^kX=K^kX$, the Rees construction is functorial in
$X\in\Pair$.  Moreover its special fiber is
\[
  R_G(C)/uR_G(C)
  \cong\bigoplus_k IF(K^kX).
\]
\end{corollary}

\begin{proof}
Corollary~\ref{cor:tower-functorial} gives
$\phi(G^kX)\subseteq G^kY$ for every morphism $\phi\colon X\to Y$, so
$xu^{-k}\mapsto\phi(x)u^{-k}$ defines the Rees morphism.  The special-fiber description
is Proposition~\ref{prop:canonical-layers}.
\end{proof}

\begin{remark}\label{rem:rees-meaning}
The Rees family is therefore not attached to an arbitrary auxiliary filtration when the
canonical tower is used: it is functorially determined by the unit of the reflection.
The reflector gives the first uncurved quotient, its kernel tower resolves the remaining
obstruction layer by layer, and the Rees module assembles those layers into a flat
one-parameter degeneration.
\end{remark}

\section{Trace Consequences}\label{sec:trace}

From here through the zeta and torsion examples we assume that all chain groups are
finite-dimensional.

\begin{lemma}[Trace additivity]\label{lem:trace-additivity}
Let $0\to U\to V\to W\to0$ be an equivariant short exact sequence of finite-dimensional
vector spaces.  Then $\tr(V)=\tr(U)+\tr(W)$ for the corresponding endomorphisms.
\end{lemma}

\begin{proof}
Choose a basis of $V$ adapted to $U$; the representing matrix is block upper triangular.
\end{proof}

\begin{theorem}[Homological representation]\label{thm:homological-representation}
For every uncurving filtration $H^\bullet$ of $(C_*,f)$,
\begin{equation}\label{eq:homological-representation}
  \str_{C_*}(f)
  =\sum_{k\geq0}
   L\bigl(H_*(\gr_H^k f),H_*(\gr_H^k C_*)\bigr).
\end{equation}
\end{theorem}

\begin{proof}
Trace additivity along
$0\to H^{k+1}_n\to H^k_n\to\gr_H^kC_n\to0$, followed by summation over $k$ and then
$n$, gives
\[
  \str_{C_*}(f)=\sum_k\str_{\gr_H^kC_*}(\gr_H^kf).
\]
Fix one layer $D_*=\gr_H^kC_*$ and write $g=\gr_H^kf$.  Since $g$ is a chain map, it
preserves cycles $Z_n$ and boundaries $B_n$.  Trace additivity in
\[
  0\to Z_n\to D_n\to B_{n-1}\to0,
  \qquad
  0\to B_n\to Z_n\to H_n(D)\to0
\]
gives
\[
  \tr(g,D_n)=\tr(g,B_n)+\tr(H_n(g))+\tr(g,B_{n-1}).
\]
The boundary terms cancel in the alternating sum over $n$, leaving
$\str_D(g)=L(H_*(g),H_*(D))$.  Applying this to every layer proves
\eqref{eq:homological-representation}.
\end{proof}

\begin{corollary}[Canonical Lefschetz representation]\label{cor:canonical-lefschetz}
For the canonical kernel tower $G^\bullet$ let
\[
  \gr^G C_*:=\bigoplus_k\gr_G^kC_*,
  \qquad
  \gr f:=\bigoplus_k\gr_G^kf.
\]
Then $(\gr^GC_*,\gr f)$ is an uncurved pair, functorial in $(C_*,f)$, and
\[
  \str_{C_*}(f)=L\bigl(H_*(\gr f),H_*(\gr^GC_*)\bigr).
\]
\end{corollary}

\begin{proof}
The layers are uncurved by Proposition~\ref{prop:canonical-layers}, functoriality is
Corollary~\ref{cor:tower-functorial}, and the trace identity is
Theorem~\ref{thm:homological-representation}.
\end{proof}

\begin{corollary}[Hopf trace formula]\label{cor:hopf-trace}
If $f$ is a chain endomorphism of $C_*$, then
\[
  \str_{C_*}(f)=L(H_*(f),H_*(C_*)).
\]
\end{corollary}

\begin{proof}
The filtration $C_*\supseteq0$ is uncurving, so the assertion is the one-layer case of
Theorem~\ref{thm:homological-representation}.
\end{proof}

\begin{corollary}[Euler--Poincar\'e]\label{cor:euler}
For every bounded finite-dimensional chain complex,
\[
  \chi(C_*)=\sum_n(-1)^n\dim H_n(C_*).
\]
\end{corollary}

\begin{proof}
Apply Corollary~\ref{cor:hopf-trace} to the identity endomorphism.
\end{proof}

\begin{corollary}[First-stage defect decomposition]\label{cor:defect-decomposition}
For every pair $(C_*,f)$,
\begin{equation}\label{eq:defect-decomposition}
  \str_{C_*}(f)
  =\str_{\Delta(C,f)}(f)
   +L\bigl(H_*(\widetilde f),H_*(Q(C,f))\bigr).
\end{equation}
For the canonical tower the first term is the sum of the Lefschetz numbers of all
layers of positive filtration degree.
\end{corollary}

\begin{proof}
Trace additivity for
$0\to\Delta(C,f)\to C_*\to Q(C,f)\to0$ gives the first term plus
$\str_{Q(C,f)}(\widetilde f)$, and Corollary~\ref{cor:hopf-trace} identifies the latter
with its Lefschetz number.  Applying Theorem~\ref{thm:homological-representation} to
$\Delta(C,f)=G^1C$ gives the final assertion.
\end{proof}

\begin{remark}\label{rem:first-stage}
The first-stage decomposition separates the universal uncurved quotient from the
unresolved kernel.  If $\str_{\Delta(C,f)}(f)=0$, the original supertrace is already the
Lefschetz number of the reflector $F(C,f)$.  Otherwise the kernel tower resolves this
remaining contribution into subsequent uncurved layers.
\end{remark}

\begin{corollary}[Trace along the Rees degeneration]\label{cor:rees-trace}
For an uncurving filtration $H^\bullet$,
\[
  \str_{R_H(C)}(f)=\str_{C_*}(f)=\str_{\gr_HC_*}(\gr_Hf).
\]
Thus the homological representation theorem is the Hopf trace formula on the uncurved
special fiber of the Rees family.
\end{corollary}

\begin{proof}
The trace of $f$ on the free $\K[u]$-module $R_H(C)_n$ equals the trace after
localization at $u$, hence equals $\tr(f_n)$; specialization at $u=0$ gives the trace on
$\gr_HC_n$.  Taking alternating sums gives the displayed equality.  The final statement
follows from Corollary~\ref{cor:hopf-trace} applied to the special fiber.
\end{proof}

\section{Examples}\label{sec:examples}

The first two finite-dimensional examples and Proposition~\ref{prop:multiplication-curvature}
are elementary test cases constructed here to exhibit the defect and kernel-tower formulas.
The remaining examples are classical constructions or applications whose provenance is
indicated explicitly below.  The zeta example packages all power supertraces
multiplicatively, while the final two examples connect the uncurving mechanism with
standard topological and geometric constructions.

\begin{example}[Nonvanishing defect contribution]\label{ex:interval}
Let $C_*$ be the cellular chain complex of the closed interval, $C_1=\K e_1$,
$C_0=\K e_0^1\oplus\K e_0^2$, $d(e_1)=e_0^2-e_0^1$, and let $f$ be the identity on
$C_1$ and the transposition of $e_0^1,e_0^2$ on $C_0$, so that
$\str_{C_*}(f)=0-1=-1$.  For odd $\ell$,
$[d,f^\ell](e_1)=2(e_0^2-e_0^1)$, while for even $\ell$ it vanishes.  Hence
$\Delta(C,f)=\K(e_0^2-e_0^1)$ in degree $0$.  The quotient has
$Q_1=\K e_1$, $Q_0=\K[e_0^1]$, $\widetilde d=0$ and $\widetilde f=\id$, so
$L(H_*(\widetilde f),H_*(Q(C,f)))=1-1=0$.  Since $f$ acts on $\Delta_0(C,f)$ as $-1$,
Corollary~\ref{cor:defect-decomposition} reads $-1=(-1)+0$.  The canonical tower is
$C_*\supset\Delta(C,f)\supset0$, of uncurving depth $2$.
\end{example}

\begin{example}[Nonzero defect, vanishing defect contribution]\label{ex:intermediate}
Let $C_1=\K e_1$, $C_0=\bigoplus_{i=1}^4\K e_0^{(i)}$,
$d(e_1)=e_0^{(1)}+e_0^{(2)}$, $f_1=\id$, and
\[
  f_0(e_0^{(1)})=e_0^{(1)},\quad
  f_0(e_0^{(2)})=e_0^{(2)}+e_0^{(3)},\quad
  f_0(e_0^{(3)})=e_0^{(4)},\quad
  f_0(e_0^{(4)})=e_0^{(3)}.
\]
Then
$[d,f](e_1)=-e_0^{(3)}$ and the $f$-stable span of this vector is
$\Delta(C,f)=\K e_0^{(3)}\oplus\K e_0^{(4)}$.  On the defect $f$ swaps the two basis
vectors, so $\str_{\Delta(C,f)}(f)=0$ although $f$ is not a chain map.  The quotient has
$Q_1=\K[e_1]$, $Q_0=\K[e_0^{(1)}]\oplus\K[e_0^{(2)}]$, $\widetilde f=\id$, and
$H_*(Q(C,f))=\K$ in degree $0$.  Hence
$L(H_*(\widetilde f),H_*(Q(C,f)))=1=\str_{C_*}(f)$.  Again the uncurving depth is $2$.
\end{example}

\begin{example}[Lefschetz zeta functions]\label{ex:lefschetz-zeta}
The classical Lefschetz zeta function packages the Lefschetz numbers of the iterates of a
self-map or chain endomorphism into a single rational function; see
\cite[p.~122]{Milnor1968Cyclic} and \cite{Franks1976}.  In the present setting, for a pair
$(C_*,f)$ define
\begin{equation}\label{eq:curved-zeta}
  \mathcal Z(C_*,f;t)
  :=\exp\!\left(\sum_{r\geq1}\frac{\str_{C_*}(f^r)}{r}t^r\right)
  =\prod_n\det(1-tf_n)^{(-1)^{n+1}}.
\end{equation}
The determinant identity follows from
$\log\det(1-tT)=-\sum_{r\geq1}\tr(T^r)t^r/r$.  If $H^\bullet$ is any uncurving
filtration, then every $f^r$ preserves it and
$\gr_H^k(f^r)=(\gr_H^kf)^r$.  Applying
Theorem~\ref{thm:homological-representation} to every power gives
\[
  \str_{C_*}(f^r)
  =\sum_k L\bigl(H_*((\gr_H^kf)^r),H_*(\gr_H^kC_*)\bigr).
\]
Hence
\begin{equation}\label{eq:zeta-example-factorization}
  \mathcal Z(C_*,f;t)
  =\prod_k\zeta(\gr_H^kf;t)
  =\zeta(\gr_Hf;t),
\end{equation}
where, for a chain endomorphism $g$ of a finite complex $D_*$,
\[
  \zeta(g;t)
  :=\exp\!\left(\sum_{r\geq1}
       \frac{L(H_*(g^r),H_*(D_*))}{r}t^r\right)
  =\prod_n\det(1-tH_n(g))^{(-1)^{n+1}}
\]
is the ordinary Lefschetz zeta function; the determinant description is the standard
homological expression for this zeta function, cf.~\cite[p.~122]{Milnor1968Cyclic} and
\cite{Franks1976}.  Thus uncurving represents the complete sequence of power supertraces
by a classical homological zeta function.  For the
identity,
\[
  \zeta(\id_C;t)=(1-t)^{-\chi(C_*)},
\]
so Euler characteristic is the exponent of the Lefschetz zeta function of the identity.
\end{example}

\begin{proposition}[Multiplication by a non-closed element]\label{prop:multiplication-curvature}
Let $A$ be a dg-algebra, $M$ a left dg-module over $A$, and $a\in A_0$.  For the
degree-zero endomorphism $f=L_a$ of $M$ one has
\[
  [d,f]=L_{da},\qquad
  \Delta(M,L_a)=\sum_{j\geq0}a^j(da)M.
\]
If $A$ is graded-commutative, then
\[
  \Delta(M,L_a)=(da)M,
\]
and $L_a$ is a chain map on the defect.  Consequently the canonical tower has
uncurving depth at most $2$.
\end{proposition}

\begin{proof}
The dg-module Leibniz rule gives $d(am)-a\,d(m)=(da)m$, hence
$[d,L_a]=L_{da}$.  Formula~\eqref{eq:defect-span} gives
\[
  \Delta(M,L_a)=\sum_{j\geq0}L_a^j((da)M)
               =\sum_{j\geq0}a^j(da)M.
\]
If $A$ is graded-commutative, the degree-zero element $a$ commutes with $da$, so the
right-hand side is $(da)M$.  Since $da$ has odd degree and $\operatorname{char}\K=0$,
$(da)^2=0$.  Thus $[d,L_a]$ vanishes on $(da)M$.
\end{proof}

\begin{example}[Banchoff's vertex index and Reidemeister torsion]\label{ex:depth-three}
Let $X$ be a finite simplicial complex with a total ordering of its vertices and let
$C_{-q}:=C^q(X;\K)$, equipped with the Alexander--Whitney cup product
\cite[Section~3.2]{Hatcher02} and the cochain differential.  For a vertex $v$, let $\chi_v\in C^0(X;\K)$ be its indicator cochain and
put $f_v:=\chi_v\smile-$.  If $\sigma=[v_0,\ldots,v_q]$ is written in increasing order,
then
\[
  (\chi_v\smile\beta)(\sigma)=\chi_v(v_0)\beta(\sigma).
\]
Thus $f_v$ is the projection onto the span of the dual simplices $\sigma^*$ with
$\min\sigma=v$.  Writing $\operatorname{lk}^+(v)$ for the part of the link of $v$ on
vertices larger than $v$, one obtains
\begin{equation}\label{eq:banchoff-index}
  \str(f_v)
  =\sum_{\min\sigma=v}(-1)^{\dim\sigma}
  =1-\chi(\operatorname{lk}^+(v)).
\end{equation}
This is Banchoff's critical-point index for a height function decreasing along the
chosen vertex order; Banchoff's critical-point theorem is
\cite[p.~246, Theorem~1]{Banchoff1967}.  Reversing the height interchanges upper and lower
links; for the now-standard lower-link formulation see
\cite[Section~VI.3, pp.~135--139]{EdelsbrunnerHarer2010}.  Since $\sum_v\chi_v=1$, one
has $\sum_vf_v=\id$, and summing \eqref{eq:banchoff-index} gives Banchoff's index formula
$\sum_v\operatorname{ind}(v)=\chi(X)$.

The defect is generated by multiplication by $\delta\chi_v$.  A direct calculation gives
\begin{equation}\label{eq:banchoff-defect}
  \Delta(C,f_v)
  =\left\langle [v,\tau]^*:\ \varnothing\neq\tau\in\operatorname{lk}^+(v)\right\rangle
   +\left\langle \beta_\sigma:\ \min\sigma=v\right\rangle,
\end{equation}
where
\[
  \beta_\sigma
  :=\sum_{\substack{u<v\\ {}[u,\sigma]\in X}}[u,\sigma]^*.
\]
Indeed $[d,f_v]=L_{\delta\chi_v}$ sends a basis cochain either to a simplex
$[v,\tau]^*$ with all vertices of $\tau$ larger than $v$, or to the corresponding sum
$\beta_\sigma$ when the basis simplex $\sigma$ begins at $v$; moreover $f_v$ fixes the
first family and kills the second, so their span is already $f_v$-stable.

On $Q(C,f_v)$, the induced map $\widetilde f_v$ is the projection onto the class of
$[v]^*$ in degree $0$ and is zero in all other degrees.  This class is a cycle because
\[
  \delta[v]^*
   =\beta_{[v]}-\sum_{\substack{w>v\\ {}[v,w]\in X}}[v,w]^*
   \in\Delta(C,f_v),
\]
and it cannot be a boundary in degree $0$.  Hence
\begin{equation}\label{eq:banchoff-splitting}
  L(H_*(\widetilde f_v),H_*(Q(C,f_v)))=1,
  \qquad
  \str_{\Delta(C,f_v)}(f_v)=-\chi(\operatorname{lk}^+(v)).
\end{equation}
Thus Corollary~\ref{cor:defect-decomposition} is precisely Banchoff's ``$1$ minus the
link'': the vertex is the homological contribution and the upper link is the defect
contribution.

The kernel tower records the noncommutativity of the cochain product.  The commutator
sends the first family in \eqref{eq:banchoff-defect} into the second and kills the
second, giving
\[
  G^2
   =\left\langle\beta_\sigma:\ \min\sigma=v,\ \dim\sigma\geq1\right\rangle,
  \qquad G^3=0.
\]
Consequently the uncurving depth is $3$ exactly when $v$ is the middle vertex of some
$2$-simplex, is $2$ when $v$ is non-isolated but is never such a middle vertex, and is
$1$ when $v$ is isolated.  For the simplex $[0,1,2]$ and $v=1$,
\[
  \str(f_1)=0,\qquad
  \Delta(C,f_1)=\langle[01]^*,[12]^*,[012]^*\rangle,\qquad
  G^2=\langle[012]^*\rangle,
\]
and the three uncurved layers contribute $1,-1,0$.

The multiplicative refinement follows from Example~\ref{ex:lefschetz-zeta}.  Since
$f_v^2=f_v$,
\begin{equation}\label{eq:banchoff-local-zeta}
  \mathcal Z(C^\bullet(X),f_v;t)
  =(1-t)^{-\operatorname{ind}(v)}
  =(1-t)^{\chi(\operatorname{lk}^+(v))-1}.
\end{equation}
The idempotents $f_v$ are pairwise orthogonal and sum to the identity, so
\begin{equation}\label{eq:banchoff-global-zeta}
  \prod_v\mathcal Z(C^\bullet(X),f_v;t)
  =(1-t)^{-\chi(X)}
  =\zeta(\id_{C^\bullet(X)};t).
\end{equation}
Thus Banchoff's additive index theorem exponentiates to the Lefschetz zeta function of
the identity.  Over $\mathbb Q$, let $\tau(X\times S^1;t)$ denote the Reidemeister
torsion of the chain complex of the infinite cyclic cover
$X\times\mathbb R\to X\times S^1$ after extension of scalars from
$\mathbb Q[t,t^{-1}]$ to $\mathbb Q(t)$.  With Milnor's normalization of the
torsion--zeta identity, this torsion is defined up to the classical unit ambiguity
$\pm t^m$ and satisfies
\begin{equation}\label{eq:banchoff-reidemeister}
  \tau(X\times S^1;t)\doteq(1-t)^{\chi(X)}
  =\prod_v(1-t)^{\operatorname{ind}(v)}
  =\prod_v\mathcal Z(C^\bullet(X),f_v;t)^{-1}.
\end{equation}
This is the multiplicative Euler-characteristic aspect of Reidemeister torsion in the
present setting.  The normalization is worth recording because reciprocal conventions
for Reidemeister torsion also occur in the literature.  Milnor proves the relevant
infinite-cyclic-cover torsion--zeta relation in
\cite[Section~3, pp.~121--123]{Milnor1968Cyclic}; see also
\cite[Introduction and Corollary~5.2]{KleinMalkiewich2022} for a modern $K$-theoretic
formulation.
\end{example}

\begin{example}[Relative de Rham complex as the uncurved quotient]\label{ex:de-rham}
The reflector, kernel tower, and Rees construction do not require finite-dimensional
chain groups.  Let $X$ be a smooth manifold and let $A=M=\Omega_X^\bullet$, regraded
homologically by $M_{-q}=\Omega_X^q$, with its de Rham differential.  For a function
$\phi\in\Omega_X^0$, let $f=L_\phi$.  Proposition~\ref{prop:multiplication-curvature}
gives
\begin{equation}\label{eq:relative-de-rham-quotient}
  \Delta(M,L_\phi)=d\phi\wedge\Omega_X^{\bullet-1},
  \qquad
  Q(M,L_\phi)
   =\Omega_X^\bullet/(d\phi\wedge\Omega_X^{\bullet-1}).
\end{equation}
In the algebraic setting, suppose now that $\phi\colon X\to S$ is a smooth morphism of
smooth varieties to a one-dimensional smooth base with local coordinate $t$ and
$\phi^*t=\phi$.  Then the quotient on the right is the usual relative de Rham complex
$\Omega_{X/S}^\bullet$; see the construction of relative de Rham cohomology in
\cite[Section~2, pp.~201--203]{KatzOda1968}.  The
induced endomorphism is multiplication by the base coordinate.  Thus the reflection
theorem says that $\Omega_{X/S}^\bullet$ is the universal quotient of
$\Omega_X^\bullet$ on which multiplication by $\phi$ becomes a chain map; in this
one-parameter situation, passing to relative differentials imposes precisely the
relation $d\phi=0$.

Independently, the same quotient construction makes sense for germs at
$0\in\mathbb C^n$, where the map defined by $\phi$ need not be smooth.  If
$\phi\in\mathcal O_{\mathbb C^n,0}$ has an isolated critical point, then in top degree
\[
  Q^n
   =\Omega^n_{\mathbb C^n,0}/d\phi\wedge\Omega^{n-1}_{\mathbb C^n,0}
   \cong
   \frac{\mathcal O_{\mathbb C^n,0}}
        {(\partial_1\phi,\ldots,\partial_n\phi)}\,dx_1\wedge\cdots\wedge dx_n,
\]
so $\dim_\mathbb C Q^n$ is the Milnor number $\mu(\phi)$; this is the classical
Jacobian (Milnor) algebra.  For the relation between the Jacobian quotient and the Milnor
number see \cite[Section~7 and Appendix~B]{Milnor1968Singular}; its role in the
Gauss--Manin description of an isolated hypersurface singularity is developed by
Brieskorn in \cite{Brieskorn1970}.  Since $(d\phi)^2=0$, the
restriction of $L_\phi$ to $G^1=d\phi\wedge\Omega_X^\bullet$ is a chain map, so the
uncurving depth is at most $2$.  Proposition~\ref{prop:rees} gives the corresponding
flat one-parameter interpolation between the absolute complex and its uncurved
associated graded; no trace assertion is made in this infinite-dimensional setting.
\end{example}

\begin{remark}[Beyond one endomorphism]\label{rem:algebra-actions}
A pair $(C_*,f)$ may be viewed as a chain complex whose underlying graded vector space
carries an action of $\K[t]$, with $t$ acting by $f$, without requiring the differential
to be $\K[t]$-linear.  The obstruction is the derivation
\[
  \delta_f\colon\K[t]\longrightarrow\End_{-1}(C_*),
  \qquad
  \delta_f(p)=[d,p(f)].
\]
More generally, for an action $\rho\colon B\to\End_0(C_*)$ of an associative algebra,
set
\[
  \Delta_B(C,\rho):=\sum_{b\in B}\im[d,\rho(b)].
\]
This is a subcomplex because $d[d,\rho(b)]=-[d,\rho(b)]d$, and it is $B$-stable by the
identity
\[
  \rho(a)[d,\rho(b)]
   =[d,\rho(ab)]-[d,\rho(a)]\rho(b).
\]
Hence $C/\Delta_B(C,\rho)$ carries an action of $B$ by chain maps.  The same kernel
argument as in Theorem~\ref{thm:reflection} shows that it is the universal quotient with
this property.  Thus the reflection theorem is the one-generator case of a broader
uncurving principle for algebra actions.  We leave the corresponding kernel tower and
Rees deformation for future work.
\end{remark}

\end{document}